\documentclass[11pt]{amsart}%
\usepackage{amsmath,amssymb,amsfonts}%
\usepackage{amsthm}%
\usepackage{mathrsfs}%
\usepackage[title]{appendix}%
\usepackage{xcolor}%
\usepackage{manyfoot}
\usepackage[mathlines]{lineno}
\usepackage[utf8]{inputenc}

\newtheorem{theorem}{Theorem}[section]
\newtheorem{lemma}{Lemma}[section]

\theoremstyle{definition}

\newtheorem{definition}[theorem]{Definition}

\theoremstyle{remark}
\theoremstyle{remark}
\newtheorem{remark}{Remark}[section]

\newtheorem{rule-def}[theorem]{Rule}

\usepackage{microtype}
\usepackage{enumerate}
\usepackage{mismath}
\usepackage{stix}
\usepackage{hyperref}

\allowdisplaybreaks

\begin{document}


\title[Boundedness of integral operators]{Boundedness of $p$-adic Hardy--Hilbert and Erdélyi--Kober fractional integral operators on $p$-adic Ces\`aro function Spaces}

\author[R. Saini]{Rishabh Saini}

\address[R. Saini]{School of Mathematical and Statistical Sciences,  Indian Institute of Technology Mandi, Kamand (H.P.), 175005,  India}
\email{d23030@students.iitmandi.ac.in}

\author[Q. Jahan]{Qaiser Jahan*}
\address[Q. Jahan]{School of Mathematical and Statistical Sciences,  Indian Institute of Technology Mandi, Kamand (H.P.), 175005,  India}
\email{qaiser@iitmandi.ac.in}

\author[S. Ashraf]{Salman Ashraf}

\address[S. Ashraf]{Department of Mathematics, Applied Science Cluster, University of Petroleum and Energy Studies (UPES), Dehradun, Uttarakhand 248007, India}
\email{ashrafsalman869@gmail.com}

\begin{abstract}

In this paper, we introduce Ces\`aro function spaces over $p$-adic fields and investigate their fundamental properties, such as the dilation operator and the Minkowski-type integral inequality. We establish boundedness result for $p$-adic Hardy--Hilbert-type integral operators acting on $p$-adic Ces\`aro function spaces, and as an application we derive $p$-adic analogue of the Hardy inequality, the Hilbert inequality, and the Hardy-Littlewood-P\'{o}lya inequality. Furthermore, we define the $p$-adic analogue of the Erd\'elyi–Kober fractional integral operators and prove their boundedness on $p$-adic Ces\`aro function spaces with the help of the obtained boundedness result.
\end{abstract}

\maketitle

\section{Introduction}

The Ces\`aro function spaces $Ces_r(I)$, defined on $I=[0,1]$ or $[0,\infty)$, constitute an important class of Banach function spaces that are not rearrangement-invariant. These spaces arise as continuous analogues of Ces\`aro sequence spaces and became popular, following the problem posed by the Dutch Mathematical Society in 1968, concerning the characterization of their dual spaces \cite{AS1}. Subsequent studies have shown that Ces\`aro spaces exhibit several non classical features: they are neither rearrangement-invariant nor reflexive, and their geometric structure differs substantially from that of Lebesgue spaces. Due to these distinctive properties, Ces\`aro function spaces provide a natural framework for the study of integral operators and for investigating sharp constants in integral inequalities. For a detailed account of their duality theory, geometric structure, and related analytical aspects, we refer to \cite{AI,AS1,AS2,AS3,AS4,AS5,KT,KL2,KL3}.

A fundamental problem in the study of Ces\`aro function spaces concerns the boundedness of integral operators acting on these spaces. In the classical setting on $[0,\infty)$, the Ces\`aro norm is closely connected with Hardy-type inequalities, which has motivated the development of general criteria for operator boundedness. In this direction, Ho \cite{KP} introduced a unified approach based on the mapping properties of the dilation operators together with Minkowski-type inequalities. This method led to boundedness results for several integral operators, including the Hilbert operator, the Erdélyi–Kober fractional integral, and the Mellin fractional transform in the Archimedean setting. 

To the best of our knowledge, Ces\`aro function spaces over the field of $p$-adic numbers have not been studied so far. Consequently, the boundedness theory of integral operators on such spaces remains unknown in the non-Archimedean setting. This naturally leads to the question of whether the classical theory of Ces\`aro spaces and the associated operator theory can be extended to the framework of $p$-adic analysis.
Unlike real numbers, $p$-adic numbers form a non-Archimedean field, where the concepts of distance and convergence differ significantly from those in the Euclidean setting. Because of this distinct structure, function spaces over $p$-adic fields often exhibit behavior that is quite different from their real counterparts. In recent years, the study of integral operators on function spaces over the $p$-adic fields has gained growing interest and has become an active area of research within $p$-adic harmonic analysis (see \cite{SQ1,SQ2,SH1,SH2, BJ1, BJ2, KH1,SA1} and references therein).

The extension of classical theory to the $p$-adic setting is not entirely straightforward. Unlike the Euclidean case, the $p$-adic norm takes only discrete values, and the field $\mathbb{Q}_p$ admits a decomposition into $p$-adic spheres
\[
S_k=\{x\in\mathbb{Q}_p:|x|_p=p^k\}, \quad k\in\mathbb{Z}.
\]
Consequently, many estimates are naturally reduced to sums over these spheres rather than integrals over a continuous range of radii. This feature is reflected in Theorem~\ref{thm:main}, where the boundedness of the operator is characterized by the discrete summability condition
\[
(1-p^{-1})\sum_{j\in\mathbb Z}\mathscr{K}(1,p^j)p^{j(1-1/r)}<\infty,
\]
where $\mathscr{K}$ is the kernel of the Hardy-Hilbert-type integral operator which is defined in Section~4.
The appearance of such a condition is closely related to the valuation structure of the $p$-adic field and plays an important role in the development of the theory presented in this paper.

In this article, we introduce Ces\`aro function spaces $\mathrm{Ces}_{r}(\mathbb{Q}_p)$ over the $p$-adic fields and examine their essential characteristics, such as the dilation operator and the Minkowski-integral inequality. As a main result, we establish the boundedness of Hardy--Hilbert-type integral operators on these spaces which we obtain by using the boundedness of dilation operators on $\mathrm{Ces}_{r}(\mathbb{Q}_p)$ together with the Minkowski integral inequality in this setting. As an application of the main result, we provide the boundedness of $p$-adic Hardy operator, $p$-adic Hilbert operator and $p$-adic Hardy--Littlewood--P\'{o}lya operator on $\mathrm{Ces}_{r}(\mathbb{Q}_p).$

Furthermore, we introduce the $p$-adic analogue of the Erd\'elyi–Kober fractional integral and as an application of the boundedness of Hardy–Hilbert-type integral operators, we establish the boundedness of the $p$-adic Erd\'elyi–Kober fractional integral on $p$-adic Ces\`aro function spaces. We refer \cite{KVS,SI} to the study of Erdélyi–Kober fractional integral and its applications in fractional calculus in detail. 

The paper is organized as follows. In Section~2, we recall the necessary preliminaries from $p$-adic analysis. In Section~3, we introduce the $p$-adic Ces\`aro function spaces and the dilation operator on $\mathbb{Q}_p$, and establish several basic properties needed for subsequent arguments. The main boundedness theorem for integral operators on $p$-adic Ces\`aro function spaces is presented in Section~4. In Section~5, we apply this general result to recover Hardy-type, Hilbert-type, and Hardy--Polya-type inequalities in the $p$-adic setting. In the final section, we study the $p$-adic Erd\'elyi--Kober fractional integral operator and prove its boundedness as a consequence of the main theorem.

\section{Preliminaries}\label{ch4_S2}

In this section, we discuss the basic definitions of $p$-adic fields which we will be using throughout the article. For detailed discussion and proofs we refer \cite{BJ3, book_impulse6, Vl}.

Let $p$ be a prime number. The field of $p$-adic numbers, denoted by $\mathbb{Q}_p$, is obtained by completing the rational numbers $\mathbb{Q}$ with respect to the $p$-adic absolute value $|\cdot|_p$.
Given a nonzero rational number $x$, one can write
\[
x = p^{\gamma}\frac{m}{n},
\]
where $\gamma\in\mathbb{Z}$ and $m,n\in\mathbb{Z}$ are not divisible by $p$. The $p$-adic absolute value is then defined by
\[
|x|_p = p^{-\gamma},
\]
and by convention $|0|_p=0$.

A characteristic feature of the $p$-adic norm is that it is non-Archimedean. In particular, it satisfies the ultrametric inequality
\[
|x+y|_p \le \max\{|x|_p,|y|_p\}, \qquad x,y\in\mathbb{Q}_p.
\]
As a consequence, whenever $|x|_p$ and $|y|_p$ are distinct, the larger one dominates and
\[
|x+y|_p = \max\{|x|_p,|y|_p\}.
\]

Every nonzero $p$-adic number admits a unique expansion of the form
\begin{equation}\label{ch4_s1}
x = p^{\gamma}\sum_{j=0}^{\infty} a_j p^{j},
\end{equation}
where $\gamma\in\mathbb{Z}$, $a_0\in\{1,\ldots,p-1\}$, and $a_j\in\{0,\ldots,p-1\}$ for $j\ge1$. The 
 series in \eqref{ch4_s1} converges with respect to the $p$-adic metric, since the terms $p^{\gamma}a_jp^j$ tend to zero in $|\cdot|_p$ as $j\to\infty$. For $a\in\mathbb{Q}_p$ and $\ell\in\mathbb{Z}$, we define the closed $p$-adic ball and the corresponding sphere by
\[
B^{\ell}(a)
=
\{x\in\mathbb{Q}_p : |x-a|_p \le p^{\ell}\},
\qquad
S^{\ell}(a)
=
\{x\in\mathbb{Q}_p : |x-a|_p = p^{\ell}\}.
\]
When the center is the origin, we simply write $B^{\ell}=B^{\ell}(0)$ and $S^{\ell}=S^{\ell}(0)$. With this notation, the field $\mathbb{Q}_p$ decomposes as
\[
\mathbb{Q}_p = \bigcup_{\ell\in\mathbb{Z}} B^{\ell},
\qquad
\mathbb{Q}_p^\times = \mathbb{Q}_p\setminus\{0\}
= \bigcup_{\ell\in\mathbb{Z}} S^{\ell}.
\]

Since $\mathbb{Q}_p$ is a locally compact Abelian group under addition, there exists a translation-invariant Haar measure on $\mathbb{Q}_p$, which we denote by $dx$. We fix the normalization by requiring that the unit ball has measure one, namely
\[
\int_{B^{0}} dx = 1.
\]
For any measurable set $E\subset\mathbb{Q}_p$, we write $|E|$ for its Haar measure. With the above normalization, the measures of balls and spheres are given by
\[
|B^{\ell}| = p^{\ell},
\qquad
|S^{\ell}| = p^{\ell}(1-p^{-1}),
\quad \ell\in\mathbb{Z}.
\]

We finally recall some basic scaling properties of the Haar measure, which are crucial for the analysis of dilation-invariant operators. For any $a\in\mathbb{Q}_p^\times$ and any integrable function  $f:\mathbb{Q}_p\to\mathbb{C}$, the change of variables $x=ay$ yields
\[
\int_{\mathbb{Q}_p} f(ax)\,dx
=
|a|_p^{-1}
\int_{\mathbb{Q}_p} f(x)\,dx.
\]
Equivalently, for any measurable set $E\subset\mathbb{Q}_p$, one has
\[
|aE| = |a|_p\,|E|,
\qquad
aE = \{ax : x\in E\}.
\]
In particular, multiplication by $a$, acts as a dilation on balls,
satisfying
\[
aB^{\ell} = B^{\,\ell+v_p(a)},
\]
where $v_p(a)$ denotes the $p$-adic valuation of $a$.

These scaling relations express the homogeneity of the Haar measure with respect to $p$-adic dilations and will be repeatedly used in the sequel, especially in the study of Hardy-Hilbert type operators, and fractional integral operators on $\mathbb{Q}_p$.

\section{\texorpdfstring{ Ces\`aro function spaces over $p$}{p}-adic field} 
In this section, we introduce the Ces\`aro function spaces over $p$-adic fields. We establish a bound for the dilation operator and Minkowski-type integral inequality in the setting of $p$-adic Ces\`aro function spaces. These foundational results play a crucial role in the proof of the main results presented in the following sections.

\begin{definition}
    Let $ 1 \leq r < \infty.$ The $p$-adic Lebesgue space  $L^r(\mathbb{Q}_p)$ consists of all measurable functions $f$ on $\mathbb{Q}_p$ such that 
\begin{align*}
\|f\|_{L^r(\mathbb{Q}_p)} = \Bigg( \int_{\mathbb{Q}_p}|f(x)|^r \di x\Bigg)^{1/r} < \infty.
\end{align*}
\end{definition}

The space $L_{\mathrm{loc}}^r(\mathbb{Q}_p)$ is defined as the set of all measurable functions $f$ on $\mathbb{Q}_p$ satisfying $\int_{K}|f(x)|^r \di x < \infty,$ for any compact subset $K$ of $\mathbb{Q}_p.$

\begin{definition} \label{Def:1} Let $1 \leq r < \infty.$ The $p$-adic Ces\`aro function spaces $\mathrm{Ces}_{r}(\mathbb{Q}_p)$ consists of all measurable functions $f$ on $\mathbb{Q}_p$ such that

\begin{align} \label{Ces:norm1}
    \|f\|_{\mathrm{Ces}_{r}(\mathbb{Q}_p)}=\|\mathscr{H}^pf\|_{L^{r}(\mathbb{Q}_p)}=\left( \int_{\mathbb{Q}_p} \left( \frac{1}{|x|_p} \int_{|y|_p \leq |x|_p} |f(y)|\, \di y \right)^r \di x\right)^{1/r} < \infty.
\end{align}
\end{definition}

For any locally integrable function $f$ on $\mathbb{Q}_p,$ the $p$-adic Hardy operator is defined as 
\begin{align} \label{op:Hardy}
\mathscr{H}^p f(x) = \frac{1}{|x|_p} \int_{|y|_p \leq |x|_p} f(y) \di y,~~~  x \in \mathbb{Q}^*_p.
\end{align}
$p$-adic Hardy inequality established in \cite{Fu} (see Corollary 2.2) yields the embedding $L^{r}(\mathbb{Q}_p) \hookrightarrow \mathrm{Ces}_{r},~~~~~ r \in (1, \infty).$

We now study the dialation operator on $p$-adic Ces\`aro function spaces. For any measurable function $f$ on $\mathbb{Q}_p$ and $a \in \mathbb{Q}^*_p,$ the dilation operator is defined as:
\begin{align*}
    (D_a f)(x)=f(ax),\qquad x\in\mathbb{Q}_p .
\end{align*}

\begin{lemma}\label{lem:first}
Let $1 \leq r < \infty$ and   $a\in\mathbb{Q}^*_p$. We have
\begin{align} 
 \|D_a f\|_{\mathrm{Ces}_{r}(\mathbb{Q}_p)}=|a|_p^{-1/r}
\|f\|_{\mathrm{Ces}_{r}(\mathbb{Q}_p)}.
\end{align}
\end{lemma}
\begin{proof}
Recall that the $p$-adic Hardy operator is given by
$$\mathscr{H}^p f(t)=\frac{1}{|t|_p} \int_{|y|_p \leq |t|_p} f(y)\, \di y,
\qquad t\in\mathbb{Q}^*_p.$$
Therefore,
$$\mathscr{H}^p(D_a f)(t)=\frac{1}{|t|_p}\int_{|y|_p \leq |t|_p} f(ay) \di y.$$

Making the substitution $u=ay$ and using $\di (ay)=|a|_p \di y$, we obtain
$$\mathscr{H}^p(D_a f)(t)=\frac{1}{|t|_p|a|_p}\int_{|u|_p \leq |a|_p|t|_p} f(u) \di u.$$

Noting that $|a||t|_p=|at|_p$, we get
\begin{align} \label{Equality1}
    \mathscr{H}^p(D_a f)(t)=\mathscr{H}^p f(at).
\end{align}

By using \eqref{Equality1}, we have
\begin{align*}
\|D_a f\|_{\mathrm{Ces}_r(\mathbb{Q}_p)} = \left(\int_{\mathbb{Q}_p}
|\mathscr{H}^p(D_a f)(x)|^r \di x \right)^{1/r}=\left(\int_{\mathbb{Q}_p}
|\mathscr{H}^p f(ax)|^r \di x\right)^{1/r}.
\end{align*}

Again, by using the change of variable $q=ax$, and $\di (ax)=|a|_p \di x$, hence
\begin{align*}
\|D_a f\|_{\mathrm{Ces}_r(\mathbb{Q}_p)}=\left(|a|_p^{-1}\int_{\mathbb{Q}_p}
|\mathscr{H}^p f(q)|^r \di q \right)^{1/r}=|a|_p^{-1/r}\left(\int_{\mathbb{Q}_p}
|\mathscr{H}^p f(q)|^r \di q \right)^{1/r}
\end{align*}

Therefore,
\begin{align*}
   \|D_a f\|_{\mathrm{Ces}_r(\mathbb{Q}_p)}=|a|_p^{-1/r} \|f\|_{\mathrm{Ces}_r(\mathbb{Q}_p)}. 
\end{align*}

\end{proof}

The following result gives the Minkowski type integral inequality on the $p$-adic Ces\`aro function spaces.

\begin{lemma} \label{lem:Second}
Let $r\in[1,\infty).$ For any measurable function $F$ on $\mathbb{Q}_p\times\mathbb{Q}_p,$ we have
\begin{align}
 \left\|\int_{\mathbb{Q}_p} F(y,\cdot) \di y\right\|_{\mathrm{Ces}_r(\mathbb{Q}_p)} \leq \int_{\mathbb{Q}_p} \|F(y,\cdot)\|_{\mathrm{Ces}_r(\mathbb{Q}_p)} \di y.
\end{align}
\end{lemma}
\begin{proof}
According to Definition \ref{Def:1} and applying Minkowski's integral inequality\cite{BJ3} (see Theorem 1.42), we get
\begin{align*}
\left\|
\int_{\mathbb{Q}_p} F(y,\cdot) \di y \right\|_{\mathrm{Ces}_r(\mathbb{Q}_p)} &=\left(\int_{\mathbb{Q}_p} \left(\frac{1}{|x|_p}\int_{|t|_p \leq |x|_p}
\left|\int_{\mathbb{Q}_p} F(y,t) \di y\right| \di t\right)^r \di x \right)^{1/r}
\\[1ex]
&\le \left(\int_{\mathbb{Q}_p} \left(\frac{1}{|x|_p}\int_{|t|_p \leq |x|_p} \int_{\mathbb{Q}_p} |F(y,t)| \di y \di t\right)^r \di x \right)^{1/r}
\\[1ex]
&= \left(\int_{\mathbb{Q}_p} \left(\int_{\mathbb{Q}_p}\frac{1}{|x|_p}\int_{|t|_p \leq |x|_p} |F(y,t)| \di t \di y\right)^r \di x \right)^{1/r} 
\\[1ex]
& \leq\int_{\mathbb{Q}_p}\left(\int_{\mathbb{Q}_p} \left(\frac{1}{|x|_p}
\int_{|t|_p \leq |x|_p} |F(y,t)|\di t \right)^r\di x\right)^{1/r}\di y
\\[1ex]
&=\int_{\mathbb{Q}_p}\|F(y,\cdot)\|_{\mathrm{Ces}_r(\mathbb{Q}_p)}\di y.
\end{align*}

\end{proof}

\section{$p$-adic Hardy--Hilbert-type integral operators on Ces\`aro function spaces}
We now turn to the study of integral operators induced by homogeneous kernel acting on $p$-adic Ces\`aro function spaces. The Hardy--Hilbert-type integral operator $\mathscr{T}^p,$ over the $p$-adic field, is defined as
\begin{align}\label{tf:op}
\mathscr{T}^pf(x) = \int_{\mathbb{Q}_p^*} \mathscr{K}\bigl(|x|_p,|y|_p\bigr) f(y) \di y,\qquad x \in \mathbb{Q}_p^*.
\end{align}
where $\mathscr{K}(x,y)$ is a non-negative measurable function $(0,\infty)\times(0,\infty)$ satisfying 
\begin{align} \label{Homoge}
    \mathscr{K}(\tau x,\tau y) = \tau^{-1}\mathscr{K}(x,y), \qquad \tau >0.
\end{align}

In \cite{4}, Li and Jin introduced the operator $\mathscr{T}^p$ and investigated its boundedness on weighted Lebesgue spaces. Later, in \cite{SH1,KH1,KH2}, the boundedness of $\mathscr{T}^p$ was further examined in several other settings, including $p$-adic block spaces, two-weighted Morrey spaces, Morrey–Herz spaces and weighted Triebel–Lizorkin spaces.

The following theorem establishes the boundedness result for $p$-adic Hardy--Hilbert-type integral operators on Ces\`aro function spaces.

\begin{theorem} \label{thm:main}
Let $1< r < \infty.$ If the kernel $\mathscr{K}$ satisfies 
\begin{equation} \label{thm:main1}
C_{p,r}:=(1-p^{-1})\sum_{j\in\mathbb{Z}}\mathscr{K}(1,p^j)\,p^{j(1-1/r)}  < \infty.
\end{equation}
Then, for any $f\in \mathrm{Ces}_r(\mathbb{Q}_p)$, we have
\begin{align*}
    \|\mathscr{T}^pf\|_{\mathrm{Ces}_r(\mathbb{Q}_p)}\leq C_{p,r}\|f\|_{\mathrm{Ces}_r(\mathbb{Q}_p)},
\end{align*}
where $\mathscr{T}^p$ is the operator in \eqref{tf:op}.
\end{theorem}
\begin{proof} 
After the change of variable $y=\xi x$ and using $\di y=|x|_p\di \xi$
\begin{align*}
\mathscr{T}^pf(x) &= \int_{\mathbb{Q}_p^*} \mathscr{K}\bigl(|x|_p,|\xi x|_p\bigr) f(\xi x) |x|_p \di \xi \\
&= \int_{\mathbb{Q}_p^*} \mathscr{K}\bigl(|x|_p,|\xi|_p|x|_p\bigr) f(\xi x) |x|_p \di \xi \\
&= \int_{\mathbb{Q}_p^*}|x|^{-1}_p \mathscr{K}\bigl(1,|\xi|_p\bigr) f(\xi x) |x|_p \di \xi \\
&= \int_{\mathbb{Q}_p^*} \mathscr{K}\bigl(1,|\xi|_p\bigr) D_\xi f( x) \di \xi \\
\end{align*}
Applying $\|\cdot\|_{\mathrm{Ces}_r(\mathbb{Q}_p)}$ and by using Lemma \ref{lem:first} and Lemma \ref{lem:Second}, we get 
\begin{align} \nonumber
  \|\mathscr{T}^pf\|_{\mathrm{Ces}_r(\mathbb{Q}_p)} &\leq \int_{\mathbb{Q}_p^*} \mathscr{K}\bigl(1,|\xi|_p\bigr) \|D_\xi f( x)\|_{\mathrm{Ces}_r(\mathbb{Q}_p)} \di \xi \\ \nonumber
   &= \int_{\mathbb{Q}_p^*} \mathscr{K}\bigl(1,|\xi|_p\bigr) |\xi|^{-1/r}_p\|f\|_{\mathrm{Ces}_r(\mathbb{Q}_p)} \di \xi \\ \label{main:eq1}
   &= \|f\|_{\mathrm{Ces}_r(\mathbb{Q}_p)}\int_{\mathbb{Q}_p^*} \mathscr{K}\bigl(1,|\xi|_p\bigr) |\xi|^{-1/r}_p \di \xi.
\end{align}
Since $\mathbb{Q}_p^*$ can be written as a disjoint union $\bigcup_{j=-\infty}^{\infty} S^j$ where $S^j=\{x\in\mathbb{Q}_p : |x|_p = p^j\}$ and using the fact that $|S^j| = p^j(1-p^{-1}),$ we have 

\begin{align} \nonumber
\int_{\mathbb{Q}_p^*} \mathscr{K}\bigl(1,|\xi|_p\bigr) |\xi|^{-1/r}_p \di \xi &= \sum_{j\in\mathbb{Z}} \int_{S^j} \mathscr{K}\bigl(1,|\xi|_p\bigr) |\xi|^{-1/r}_p \di \xi\\ \nonumber
&= \sum_{j\in\mathbb{Z}} \mathscr{K}\bigl(1,p^j\bigr) p^{-j/r} |S^j|\\ \label{main:eq2}
&=(1-p^{-1})\sum_{j\in\mathbb{Z}} \mathscr{K}(1,p^j)\,p^{j\left(1-\frac{1}{r}\right)}.
\end{align}
Therefore, by \eqref{main:eq1} and \eqref{main:eq2}, we get
\begin{align*}
    \|\mathscr{T}^pf\|_{\mathrm{Ces}_r(\mathbb{Q}_p)}\leq C_{p,r}\|f\|_{\mathrm{Ces}_r(\mathbb{Q}_p)},
\end{align*}
which completes the proof.
\end{proof}

\begin{remark}
The condition
\[
(1-p^{-1})\sum_{j\in\mathbb Z}\mathscr{K}(1,p^j)p^{j(1-1/r)}<\infty
\]
arises naturally from the decomposition of $\mathbb{Q}_p$ into $p$-adic spheres. Since the $p$-adic norm takes only discrete values, estimates involving the kernel $\mathscr{K}$ reduce to summability conditions over these spheres. In contrast, analogous boundedness criteria in the classical setting are typically formulated in terms of integral conditions on the kernel.
Furthermore, Theorem~\ref{thm:main} provides a unified criterion for the boundedness of a broad class of integral operators on $\mathrm{Ces}_r(\mathbb{Q}_p)$. As applications of this result, we establish the boundedness of the $p$-adic Hardy operator, the $p$-adic Hilbert operator, the Hardy--Littlewood--P'olya operator, and the Erd'elyi--Kober fractional integral operator on $\mathrm{Ces}_r(\mathbb{Q}_p)$.
\end{remark}

\section{Applications of $p$-adic Hardy--Hilbert-type integral operators}
In this section, we establish the Hardy inequality, the Hilbert inequality, and the Hardy--Littlewood--P\'{o}lya inequality on $p$-adic Ces\`aro function spaces by choosing a particular kernel in \eqref{tf:op}.

If we choose the kernel $\mathscr{K}(|x|_p,|y|_p)=|x|_p^{-1}\chi_{\{y\in\mathbb{Q}_p: |y|_p \leq |x|_p \}},$ then the operator $\mathscr{T}^p$ reduces to the $p$-adic Hardy operator defined in \eqref{op:Hardy}. Hence, we have the following $p$-adic Hardy inequality for $\mathrm{Ces}_r(\mathbb{Q}_p).$

\begin{theorem}
If $r>1$, then for any $f \in \mathrm{Ces}_r(\mathbb{Q}_p)$, we have
\begin{align*}
\|\mathscr{H}^p f\|_{\mathrm{Ces}_r(\mathbb{Q}_p)} \leq
\frac{1-p^{-1}}{1-p^{-(1-1/r)}} \|f\|_{\mathrm{Ces}_r(\mathbb{Q}_p)}.
\end{align*}

\end{theorem}
\begin{proof}
Let $\mathscr{K}(|x|_p,|y|_p)=|x|_p^{-1}\chi_{\{\,y\in\mathbb{Q}_p:\,|y|_p\le |x|_p\,\}}.$ Clearly, kernal satisfies the homogeneity condition \eqref{Homoge} and
\begin{align*}
    C_{p,r}&=(1-p^{-1})\sum_{j\in\mathbb{Z}}\mathscr{K}(1,p^j) p^{j(1-1/r)} \\
    &=(1-p^{-1})\sum_{j\leq 0}p^{j(1-1/r)} \\
   & =(1-p^{-1})\sum_{j=0}^{\infty}p^{-j(1-1/r)} \\
   &= \frac{1-p^{-1}}{1-p^{-(1-1/r)}} < \infty, \qquad \text{since} \quad r>1.
\end{align*}
Therefore, result follows from Theorem \ref{thm:main}.
\end{proof}

\begin{remark}
It is worth noting that the constants obtained in the above inequalities
differ from those in the classical Archimedean setting studied by Ho \cite{KP}. This difference arises from the distinct geometry of $\mathbb{Q}_p$, especially from the ultrametric property, and the discrete nature of the $p$-adic norm. 

\end{remark}

If we choose the kernal
\begin{align} \label{Hilbert}
 \mathscr{K}(|x|,|y|)=\frac{1}{|x|_p+|y|_p}
\end{align}
in \eqref{tf:op}, then the operator $\mathscr{T}^p$ reduces to the $p$-adic Hilbert operator. More precisely, for a measurable function $f$ defined on $\mathbb{Q}_p$, we
have
\begin{align*}
    \mathscr{H}f(x)=\int_{\mathbb{Q}_p^*}\frac{f(y)}{|x|_p+|y|_p}\,\di (x),
\qquad x\in\mathbb{Q}_p^*.
\end{align*}
In the following result, we establish  Hilbert-type inequality for $p$-adic Ces\`aro function spaces $\mathrm{Ces}_r(\mathbb{Q}_p)$.

\begin{theorem}
    Let $1< r < \infty$. For any $f \in \mathrm{Ces}_r(\mathbb{Q}_p)$, we have
\begin{align*}
    \|\mathscr{H}f\|_{\mathrm{Ces}_r(\mathbb{Q}_p)}
\le C_{p,r} \|f\|_{\mathrm{Ces}_r(\mathbb{Q}_p)}.
\end{align*}
\end{theorem}

\medskip

\begin{proof}
Observe that the kernel \eqref{Hilbert} is homogeneous of degree -1 and 
\begin{align*}
C_{p,r}
&=
(1-p^{-1})
\sum_{j\in\mathbb{Z}}
\mathscr{K}(1,p^j)\,p^{j\left(1-\frac{1}{r}\right)}
\\
&=
(1-p^{-1})
\sum_{j\in\mathbb{Z}}
\frac{p^{j\left(1-\frac{1}{r}\right)}}{1+p^j}.
\end{align*}

According to \eqref{thm:main}, It is enough to show that $C_{p,r}<\infty$. Now, we can write the sum as

\begin{align*}
C_{p,r}
&= (1-p^{-1})\left(\sum_{j>0} \frac{p^{j(1-\frac{1}{r})}}{1+p^{j}}+
\sum_{j\le 0} \frac{p^{j(1-\frac{1}{r})}}{1+p^{j}}\right) \\
&= (1-p^{-1})\left(\sum_{j=1}^{\infty} \frac{p^{j(1-\frac{1}{r})}}{1+p^{j}}+\sum_{j=-\infty}^{0} \frac{p^{j(1-\frac{1}{r})}}{1+p^{j}}
\right) \\
&= (1-p^{-1})\left(\sum_{j=1}^{\infty} \frac{p^{j(1-\frac{1}{r})}}{1+p^{j}}+\sum_{j=0}^{\infty} \frac{p^{-j(1-\frac{1}{r})}}{1+p^{-j}}
\right) \\
&= (1-p^{-1})\left(\sum_{j=1}^{\infty} \frac{p^{j(1-\frac{1}{r})}}{1+p^{j}}+\sum_{j=0}^{\infty} \frac{p^{j/r}}{1+p^j}
\right) \\
&\le (1-p^{-1})\left(\sum_{j=1}^{\infty} \frac{p^{j(1-\frac{1}{r})}}{p^{j}}+\sum_{j=0}^{\infty} \frac{p^{j/r}}{p^j}
\right) \\
& \le (1-p^{-1})\left(\sum_{j=1}^{\infty} p^{-j/r}+\sum_{j=0}^{\infty} p^{-j(1-\frac{1}{r})}\right).
\end{align*}

Both series converge for  $r>1$, and hence  $C_{p,r}<\infty$. Hence, using Theorem \eqref{thm:main}, we obtain the Hilbert inequality  on $\mathrm{Ces}_r(\mathbb{Q}_p)$.
\end{proof}

 Finally, as an application of the main boundedness theorem, we study the
operator $\mathscr{D}^p_\lambda$ and establish its boundedness on
$\mathrm{Ces}_r(\mathbb{Q}_p)$. Consider the kernel in \eqref{tf:op}

\begin{align} \label{K:Polya}
    \mathscr{K}(|x|_p,|y|_p)=\frac{(|x|_p|y|_p)^{\lambda/2}} {\max\{|x|_p,|y|_p\}^{\lambda+1}}, \qquad \lambda \ge 0.
\end{align}

The corresponding integral operator is given by
\begin{equation}\label{ch4_polya}
\mathscr{D}^p_\lambda f(x) =\int_{\mathbb{Q}_p^*}\frac{(|x|_p|y|_p)^{\lambda/2}}
{\max\{|x|_p,|y|_p\}^{\lambda + 1}} f(y)\,\di y,\qquad x\in\mathbb{Q}_p^* .
\end{equation}
In the special case $\lambda=0$, the operator $\mathscr{D}^p_\lambda$
coincides with the $p$-adic Hardy--Littlewood--P\'{o}lya operator
$\mathscr{P}^p$, which is defined by
\begin{align*}
\mathscr{P}^p f(x)=\int_{\mathbb{Q}_p^*}\frac{f(y)}{\max\{|x|_p,|y|_p\}} \di y, \qquad x\in\mathbb{Q}_p^*.
\end{align*}

\begin{theorem}
Suppose $1 < r<\infty$ and $\lambda/2 +1 > 1/r$. Then the operator
$\mathscr{D}^p_\lambda$, defined by \eqref{ch4_polya}, is bounded on
$\mathrm{Ces}_r(\mathbb{Q}_p)$.
\end{theorem}
\begin{proof}
 As in the proof of Theorem \ref{thm:main}, it is sufficient to verify the finiteness of the constant $C_{p,r}$.

For $a\in\mathbb{Q}_p^*$, we have
\[
\mathscr{K}(|as|_p,|at|_p)
=
\frac{(|a|_p|s|_p\,|a|_p|t|_p)^{\lambda/2}}
{\max\{|a|_p|s|_p,|a|_p|t|_p\}^{\lambda+1}}
=
|a|_p^{-1}\mathscr{K}(|s|_p,|t|_p),
\]
which shows that the  kernal \eqref{K:Polya}  satisfies the required homogeneity condition\eqref{Homoge}.

Next, we compute
\begin{align*}
C_{p,r}=(1-p^{-1})\sum_{j\in\mathbb{Z}}\mathscr{K}(1,p^j)\,p^{j\left(1-\frac{1}{r}\right)}.
\end{align*}
A direct computation yields
\begin{align*}
\mathscr{K}(1,p^j)=\frac{p^{j\lambda/2}}{\max\{p^j,1\}^{\lambda+1}}=
\begin{cases}
p^{-j(\lambda/2+1)}, & j\ge0,\\[4pt]
p^{j\lambda/2}, & j<0.
\end{cases}
\end{align*}
Therefore,
\begin{align*}
C_{p,r}&=(1-p^{-1})\left(\sum_{j \ge 0}p^{-j(\lambda/2+1)}\,p^{j(1-\frac{1}{r})}+\sum_{j<0}p^{j\lambda/2}\,p^{j(1-\frac{1}{r})}\right)\\
&=(1-p^{-1})\left(\sum_{j\ge0}p^{-j(\lambda/2+\frac{1}{r})}+\sum_{j<0}
p^{j(\lambda/2+1-\frac{1}{r})}
\right).
\end{align*}
Since the first summation is a geometric series, it converges immediately. For the second summation, the convergence is ensured by the condition $\lambda/2 +1 > 1/r$, which guarantees that the corresponding exponent is strictly negative. Hence, the series is summable, and consequently we obtain $C_{p,r}<\infty$. The desired boundedness then follows directly from Theorem~\eqref{thm:main}.

\end{proof}
\begin{remark}
    Setting $\lambda=0$ in the above theorem, the operator $\mathscr{D}^p_\lambda$ reduces to the $p$-adic Hardy--Littlewood--P\'{o}lya operator $\mathscr{P}^p$, and consequently $\mathscr{P}^p$ is bounded on the Ces\`aro function spaces $\mathrm{Ces}_r(\mathbb{Q}_p)$. 
\end{remark}

\section{$p$-adic Fractional integrals}

In this section, we introduce the $p$-adic analogue of the Erd\'elyi--Kober fractional integral operators of both the first and second kind. As a direct consequence of Theorem~\ref{thm:main}, we establish the boundedness of these newly defined $p$-adic fractional integral operators on $\mathrm{Ces}_r(\mathbb{Q}_p)$. For the classical definition and standard notation of the Erd\'elyi--Kober fractional integral operators, we refer to \cite{KP}.

Let $\nu>0$, $\delta>0$, and $\gamma\in \mathbb{R}$. For a
measurable function $f$ on $\mathbb{Q}_p$, the first and second kind $p$-adic Erd\'elyi--Kober fractional integrals, respectively, are defined as follows:
\begin{align*}
I_{\gamma,\delta}^{\nu}f(s)=\int_{\mathbb{Q}_p^*}I(|s|_p,|t|_p) f(t) \di t, \qquad  s \in \mathbb{Q}_p^*, 
\end{align*}
and
\begin{align*}
J_{\gamma,\delta}^{\nu}f(s)=\int_{\mathbb{Q}_p^*}J(|s|_p,|t|_p) f(t) \di t, \qquad t\in\mathbb{Q}_p^*,
\end{align*}
where the kernels $I(|s|_p,|t|_p)$ and $J(|s|_p,|t|_p)$ are given by
\begin{align}
I(|s|_p,|t|_p)=\frac{|s|_p^{-\nu(\gamma+1)}}{\Gamma(\delta)}\chi_{\{t \in \mathbb{Q}_p : |t|_p < |s|_p\}}(t)
|t|_p^{\nu(\gamma+1)-1}
\left(|s|_p^\nu-|t|_p^\nu\right)^{\delta-1},
\end{align}
and
\begin{align}
J(|s|_p,|t|_p)=\frac{|s|_p^{\nu\gamma}}{\Gamma(\delta)}
\,\chi_{\{t \in \mathbb{Q}_p : |t|_p > |s|_p\}}(t)|t|_p^{-\nu(\gamma+\delta)+\nu-1}
\left(|t|_p^\nu-|s|_p^\nu\right)^{\delta-1}.   
\end{align}
The Erd\'elyi--Kober fractional integrals are among the classical operators in fractional calculus and occur in various contexts of mathematical physics. For a detailed discussion of their properties and applications, we refer \cite{KVS,SI}.
In the following theorem, we prove that the $p$-adic Erd\'elyi--Kober fractional integral is bounded on the Ces\`aro function space
$\mathrm{Ces}_r(\mathbb{Q}_p)$.

\medskip

\begin{theorem} Let $r\in(1,\infty)$ and $ \delta,\nu \in \mathbb{R}_{>0}, ~  \gamma \in \mathbb{R}$.

\begin{itemize}
\item[(1)]
If $\nu(\gamma+1)>\frac{1}{r}$, then there exists a constant $C>0$ such that
for any $f\in \mathrm{Ces}_r(\mathbb{Q}_p)$,
\[
\|I_{\gamma,\delta}^{\nu} f\|_{\mathrm{Ces}_r(\mathbb{Q}_p)}
\le
C\,\|f\|_{\mathrm{Ces}_r(\mathbb{Q}_p)}.
\]

\item[(2)]
If $\nu\gamma>-\frac{1}{r}$, then there exists a constant $C>0$ such that
for any $f\in \mathrm{Ces}_r(\mathbb{Q}_p)$,
\[
\|J_{\gamma,\delta}^{\nu}  f\|_{\mathrm{Ces}_r(\mathbb{Q}_p)}
\le
C\,\|f\|_{\mathrm{Ces}_r(\mathbb{Q}_p)}.
\]
\end{itemize}
\end{theorem}
\medskip

\noindent

\begin{proof} 
Let us sketch the proof, since 
\begin{align*}
I_{\gamma,\delta}^{\nu}f(s)=\int_{\mathbb{Q}_p^*}I(|s|_p,|t|_p) f(t) \di t, \qquad  s \in \mathbb{Q}_p^*, 
\end{align*}
where
\begin{align*}
I(|s|_p,|t|_p)=\frac{|s|_p^{-\nu(\gamma+1)}}{\Gamma(\delta)}\chi_{\{t \in \mathbb{Q}_p : |t|_p < |s|_p\}}(t)
|t|_p^{\nu(\gamma+1)-1}
\left(|s|_p^\nu-|t|_p^\nu\right)^{\delta-1}.
\end{align*}
For any $a \in \mathbb{Q}_p^*,$ we have
\begin{align*}
I(|as|_p,|at|_p) &=\frac{|as|_p^{-\nu(\gamma+1)}}{\Gamma(\delta)}\chi_{\{t \in \mathbb{Q}_p : |t|_p < |s|_p\}}(at)
|at|_p^{\nu(\gamma+1)-1}
\left(|as|_p^\nu-|at|_p^\nu\right)^{\delta-1} \\
&=|a|_p^{-1} I(|s|_p,|t|_p).
\end{align*}
Thus, the kernel $I(|s|_p,|t|_p)$ satisfies the homogeneity condition \eqref{Homoge}. Furthermore, by Theorem \ref{thm:main}, it suffices to verify that $C_{p,r}<\infty$ in order to establish the boundedness of $I_{\gamma,\delta}^{\nu} $ on $\mathrm{Ces}_r(\mathbb{Q}_p).$ Hence 
\begin{align*}
C_{p,r}&=(1-p^{-1})\sum_{j\in\mathbb{Z}}I(1,p^j)\,p^{j\left(1-\frac{1}{r}\right)}\\
&=\frac{(1-p^{-1})}{\Gamma(\delta)}
\sum_{j < 0}
p^{j\left(\nu(\gamma+1)-\frac{1}{r}\right)}
\left(1-p^{j\nu}\right)^{\delta-1}.
\end{align*}

Setting $j=-m$, we obtain
\begin{align*}
    C_{p,r} = \frac{(1-p^{-1})}{\Gamma(\delta)}\sum_{m=1}^{\infty}
p^{-m\left(\nu(\gamma+1)-\frac{1}{r}\right)}
\left(1-p^{-m\nu}\right)^{\delta-1}.
\end{align*}
Observe that for large $m$, the term
$\left(1-p^{-mh}\right)^{\delta-1}\sim 1$.
Therefore,
\begin{align*}
    C_{p,r}
=
\frac{(1-p^{-1})}{\Gamma(\delta)}
\sum_{m=0}^{\infty}
p^{-m\left(\nu(\gamma+1)-\frac{1}{r}\right)} < \infty, \qquad \text{since} ~\nu(\gamma+1)>\frac{1}{r}.
\end{align*}

\medskip

Similarly, the kernel
\begin{align}
J(|s|_p,|t|_p)=\frac{|s|_p^{\nu\gamma}}{\Gamma(\delta)}
\,\chi_{\{t \in \mathbb{Q}_p : |t|_p > |s|_p\}}(t)|t|_p^{-\nu(\gamma+\delta)+\nu-1}
\left(|t|_p^\nu-|s|_p^\nu\right)^{\delta-1}, 
\end{align}
satisfies the homogeneity condition \eqref{Homoge}. Moreover,
\begin{align*}
C_{p,r} &=(1-p^{-1})\sum_{j\in\mathbb{Z}}J(1,p^j)p^{j\left(1-\frac{1}{r}\right)} \\
&= \frac{(1-p^{-1})}{\Gamma(\delta)}\sum_{j>0}(p^{j\nu}-1)^{\delta-1}
p^{j\left(-\nu(\gamma+\delta)+\nu-\frac{1}{r}\right)}.
\end{align*}

Observe that for large $j, ~p^{j\nu}-1\sim p^{j\nu}$, so we have
\begin{align*}
C_{p,r}&=\frac{(1-p^{-1})}{\Gamma(\delta)}\sum_{j=1}^{\infty}p^{j\nu(\delta-1)}
p^{j\left(-\nu(\gamma+\delta)+\nu-\frac{1}{r}\right)}\\
&=\frac{1-p^{-1}}{\Gamma(\delta)}\sum_{j=1}^{\infty}
p^{-j\left(\nu\gamma+\frac{1}{r}\right)}
<\infty, \qquad \text{since} \qquad \nu\gamma>-\frac{1}{r}.
\end{align*}
Hence, by Theorem \eqref{thm:main}, we have the boundedness of $J_{\gamma,\delta}^{\nu}$ on $\mathrm{Ces}_r(\mathbb{Q}_p),$
which completes the proof.
\end{proof}


\bibliographystyle{amsplain}

\end{document}